\RequirePackage{fix-cm}
\documentclass[smallextended]{svjour3}       
\smartqed  
\usepackage{graphicx}

\usepackage{latexsym,amscd,amssymb} 

\journalname{}
\begin{document}

\title{Toric cubes\thanks{AE was supported by the Miller Institute at UC Berkeley.
PH was supported by NSF grant DMS-1002636 and the Ruth
Michler Prize of the Association for Women in Mathematics.
BS was supported by NSF grants DMS-0757207 and DMS-0968882.
We are grateful to Saugata Basu, Louis Billera and Rainer Sinn for helpful communications.}
}


\author{Alexander Engstr\"om \and Patricia Hersh \and Bernd Sturmfels}


\institute{
Alexander Engstr\"om \at
Department of Mathematics \\
Aalto University \\
 P.O. Box 11100 \\
FI-00076 Aalto, Finland \\
\email{alexander.engstrom@aalto.fi}
\and
Patricia Hersh \at
Department of Mathematics \\ 
     North Carolina State University \\
     Box 8205 \\
      Raleigh,  NC 27605, USA \\
\email{plhersh@ncsu.edu} 
\and
Bernd Sturmfels\at
Department of Mathematics \\
University of California  \\ Berkeley, CA 94720, USA\\ 
\email{bernd@math.berkeley.edu}
}

\date{ }

\maketitle

\begin{abstract}
A toric cube is a subset of the standard cube  defined by
binomial inequalities.
These basic semialgebraic sets are precisely the 
images of standard cubes under monomial maps. We study toric cubes
 from the perspective of
topological combinatorics. Explicit decompositions as
CW-complexes are constructed. Their
 open cells are interiors of toric cubes and 
their boundaries are subcomplexes. 
The motivating example of a  toric cube
is the edge-product space in phylogenetics, and
our work generalizes results known for that space.

\end{abstract}

\section{Introduction.}

The standard $n$-dimensional cube $[0,1]^n$ is a commutative monoid  under 
coordinatewise multiplication. In this article, we examine the 
natural class of submonoids of that monoid described next.
A {\em binomial inequality} has the form
\begin{equation}
\label{eq:binoineq}
x_1^{u_1} x_2^{u_2} \cdots x_n^{u_n} \,\, \leq \,\, 
      x_1^{v_1} x_2^{v_2} \cdots x_n^{v_n} , 
\end{equation}
where the $u_i$ and $ v_j$ are non-negative integers.
A {\em toric precube} is a subset $\mathcal{C}$ of the cube $[0,1]^n$ that is 
defined by a finite collection of binomial inequalities. 
A {\em toric cube} $\mathcal{C}$ is
toric precube that coincides with the closure of its strictly positive points.
Equivalently, a toric cube is a subset $\mathcal{C}$ of the standard cube $[0,1]^n$ that is
defined by binomial inequalities (\ref{eq:binoineq}) and also satisfies
 $\, \mathcal{C} \,=\, \overline{\mathcal{C} \,\cap\,(0,1]^n}$.
Example \ref{pre-cube-example} illustrates the difference between a toric precube
and a toric cube.

By definition, every toric cube is a basic closed  semialgebraic set in  $\mathbb{R}^n$.
Thus, the present article can be regarded as a case study in  real algebraic
geometry \cite{BCR}, with focus on a class of highly structured combinatorial objects.

Our first result is a parametric representation of toric cubes. We fix $n$ monomials
${\bf t}^{a_1},{\bf t}^{a_2},$ $ \ldots,{\bf t}^{a_n}$ in $d$
unknowns $t_1,t_2,\ldots,t_d$. The representing map
 is a monoid homomorphism from the $d$-cube to the $n$-cube:
\begin{equation}
\label{eq:alphamap}
 f \,:\, [0,1]^d \mapsto [0,1]^n \,,
\,\,(t_1,\ldots,t_d) \mapsto ({\bf t}^{a_1},{\bf t}^{a_2},\ldots,{\bf t}^{a_n}) .
\end{equation}
Our first result states that the image of any such
 monomial map of cubes is a toric cube
and, conversely, all toric cubes admit such a parametrization:

\begin{theorem}  \label{thm:one}
The toric cubes $\,\mathcal{C}$ in $[0,1]^n$ are precisely
the images of other cubes $[0,1]^d$, for any positive integer $d$, under the 
 monomial maps  $f$ into $[0,1]^n$.
\end{theorem}

\begin{example} \label{ex:ex1}
Let $n=d=3$.
The image of the monomial map 
\[  f \,: \,[0,1]^3 \rightarrow [0,1]^3 \, , \,\,
(x,y,z)\mapsto (xy, yz, xz) = (a,b,c)\] 
is the three-dimensional toric cube depicted in Figure \ref{fig:zwei}.
 It consists of all  points $(a,b,c)$ in $ [0,1]^3$
that satisfy the three  inequalities $a\ge bc, b\ge ac $ and $ c\ge ab$.
\end{example}

The intersection of a toric cube with a coordinate subspace 
is homeomorphic to a convex polyhedral cone. This is seen by taking the 
logarithm of the positive coordinates. The collection of cells 
from these polyhedral cones glue together, but sometimes further 
refinements are required to avoid the situation in Example
\ref{pre-cube-example} when only a part of a boundary cell
is glued onto  
a higher-dimensional open~cell. 

It is a problem of topological combinatorics to carry out
these refinements in a systematic manner.
The resolution
of this problem is our second result:

\begin{theorem} \label{thm:two}
Every toric cube can be realized as a CW-complex whose open cells are 
interiors of toric cubes. This CW complex has the further property that the boundary of 
each open cell is a subcomplex.
\end{theorem}

\begin{figure} 
\qquad
\includegraphics[width=10cm]{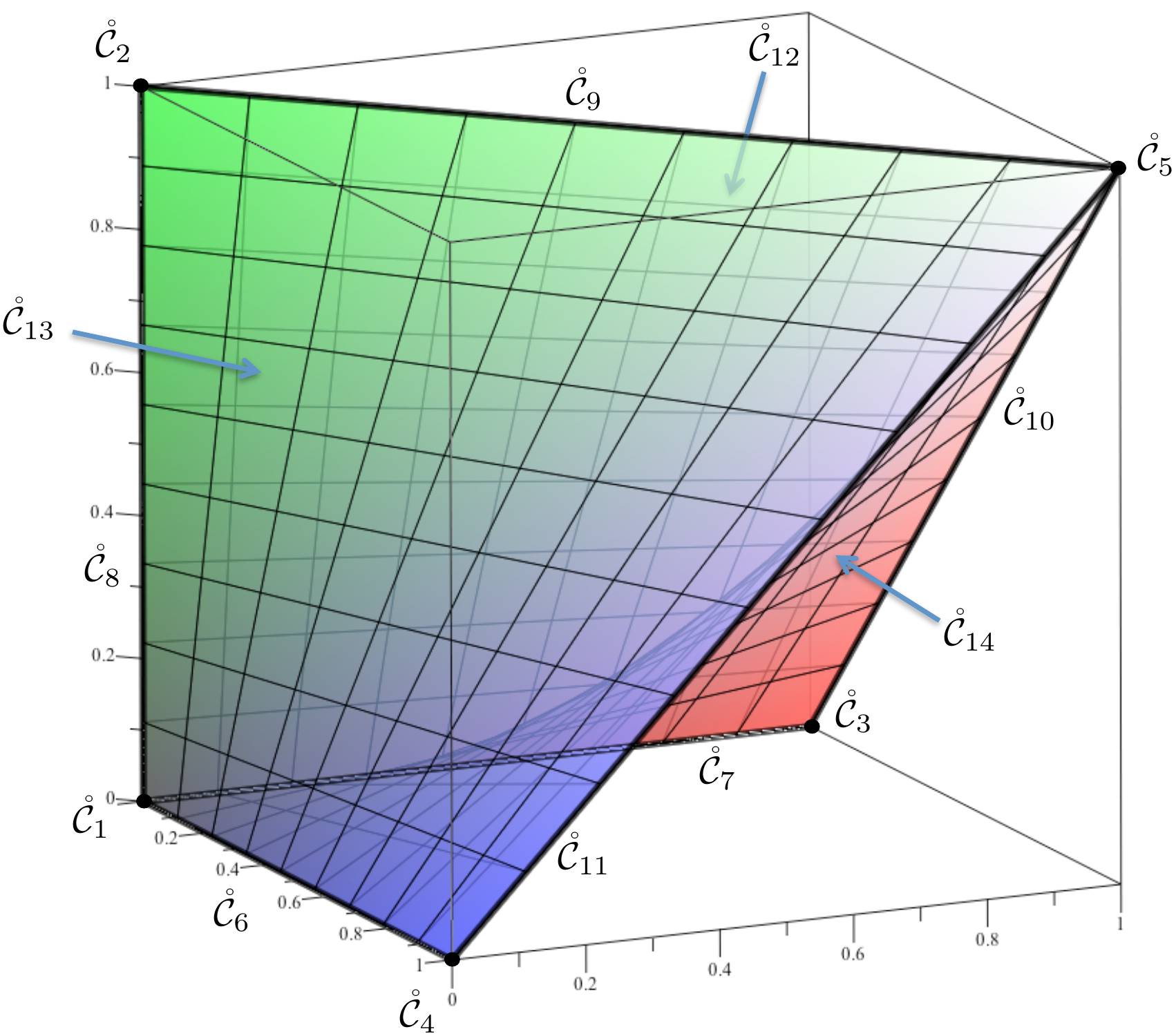}
\caption{The toric cube in Example \ref{ex:ex1} with its 
CW cell decomposition as in Example \ref{ex:toricCube}.}
\label{fig:zwei}
\end{figure}

\begin{figure}
\qquad \qquad \qquad \qquad \qquad
\includegraphics[width=6.5cm]{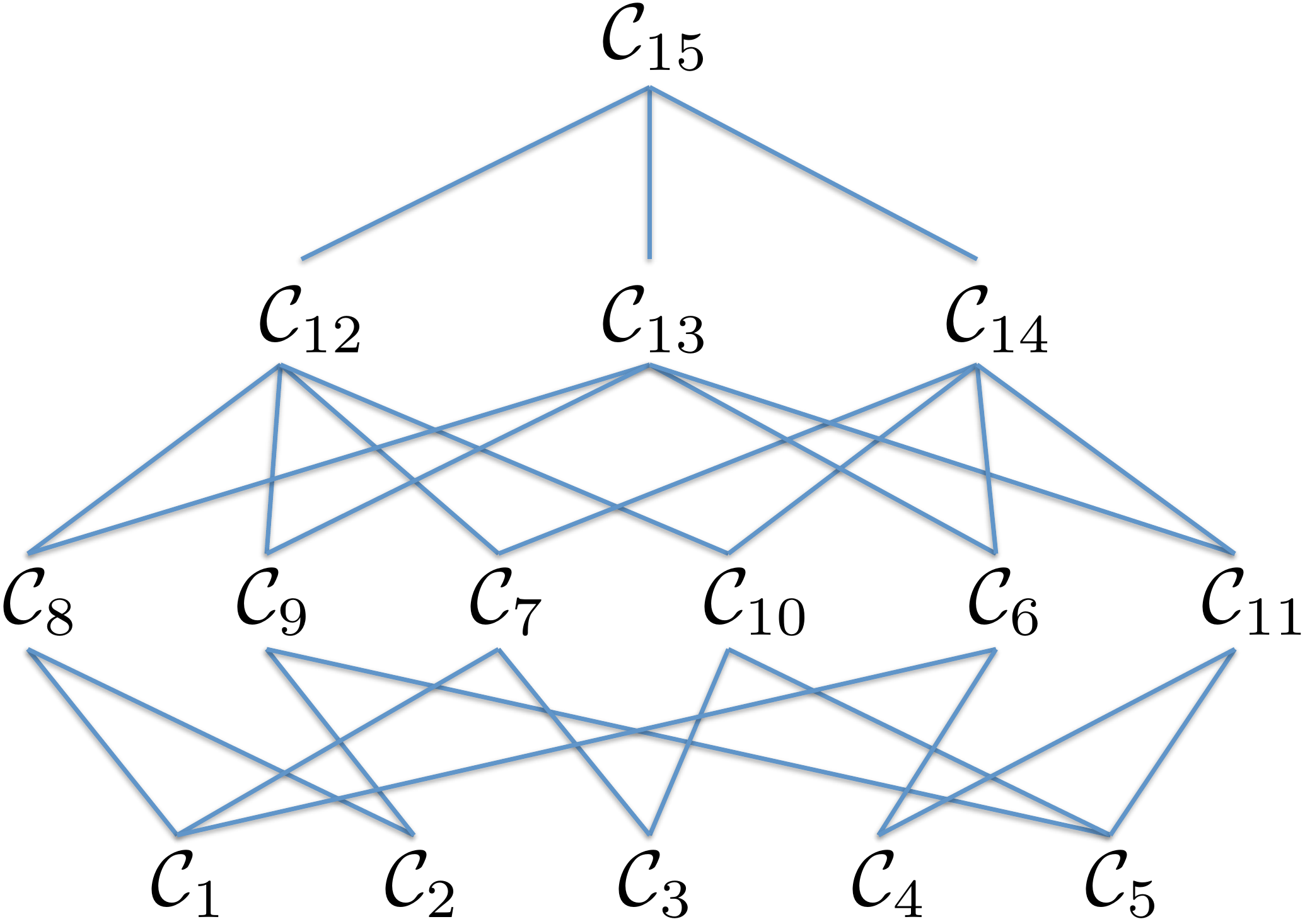}
\caption{The face poset of the CW complex defined in 
Example \ref{ex:toricCube} and drawn in Figure~\ref{fig:zwei}.}
\label{fig:drei}
\end{figure}

Our study of toric cubes is in part motivated by phylogenetics.
In that context, the topological spaces studied by Kim \cite{Kim} 
and Gill, Linusson, Moulton and Steel \cite{GLMS,MS}, are 
built from particularly well-behaved instances of toric cubes.
The toric cube in Figure \ref{fig:zwei} is the smallest instance of this:
it is the edge-product space of a tree with three leaves, as shown in \cite[Figure 1]{GLMS}.
Note that we have redrawn the same poset in Figure \ref{fig:drei} 
to indicate our cell labeling.

In those papers, every point of a toric cube corresponds to
parameters of a statistical model that describe a
phylogenetic tree. The observed data can be encoded as
points in the standard cube from which the toric cube is realized.
To infer the parameters of the model  amounts to 
running a maximum likelihood algorithm to locate the point that in
some sense is closest to the observed ones. 
Standard algorithms, such as Expectation-Maximization,
behave badly on closed sets if the desired point is reached on
the boundary. Thus, it is advantageous to break the
space into open cells and restrict to the
relevant pieces upon which one could get the algorithms to function
well. For spaces where the points correspond to random
models, there is often a first natural stratification
by combinatorial type; in the case of phylogenetic trees, each
open cell corresponds to a certain tree topology and the points in
that cell describe the lengths of  the branches.
It may be speculated that our construction here
may prove to be useful also for other models in
algebraic statistics \cite{DSS}.

Cellular complexes that are described only by their
combinatorial incidence relations, like simplicial complexes,
are called {\em regular}. Gill, Linusson, Moulton and 
Steel \cite{GLMS}  proved that the specific toric cubes
arising in the  above phylogenetic context are
regular. In the first {\tt arXiv} version of this paper, we
conjectured that all complexes built
from toric cubes share this property.
Our conjecture was subsequently proved
by Basu, Gabrielov and Vorobjov, using
the theory of monotone semi-algebraic maps
they had developed in \cite{BGV,BGV2}.
They established:

\begin{theorem}[Basu, Gabrielov and Vorobjov, Section 6 of \cite{BGV2}]
Every toric cube is homeomorphic to a
closed ball,  and every CW-complex
whose closed cells are toric cubes
is regular. In particular, the
CW-complexes for toric cubes constructed in the proof of
 Theorem~Ê\ref{thm:two}~are~regular.
\end{theorem}

\section{Parametrization and Implicitization}

Toric cubes are  objects of real algebraic geometry
that have a very nice combinatorial structure.
In particular, they are basic semialgebraic sets,
that is, they can be defined by conjunctions of polynomial inequalities.

\begin{figure}
\begin{center}
\begin{tabular}{cc}
\includegraphics[width=4.5cm]{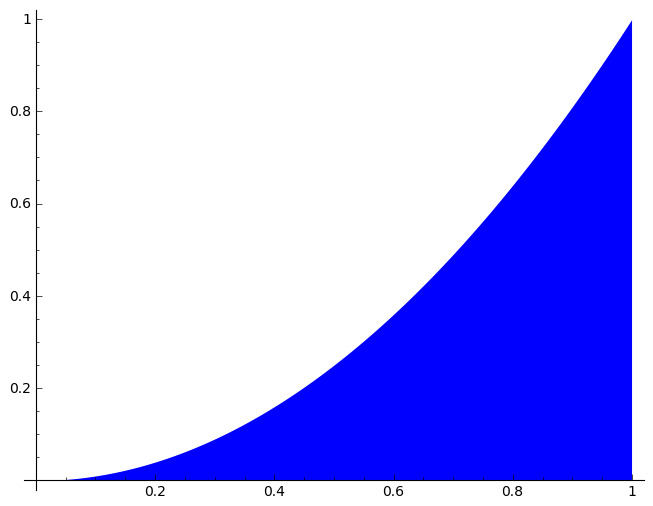} & \includegraphics[width=4.5cm]{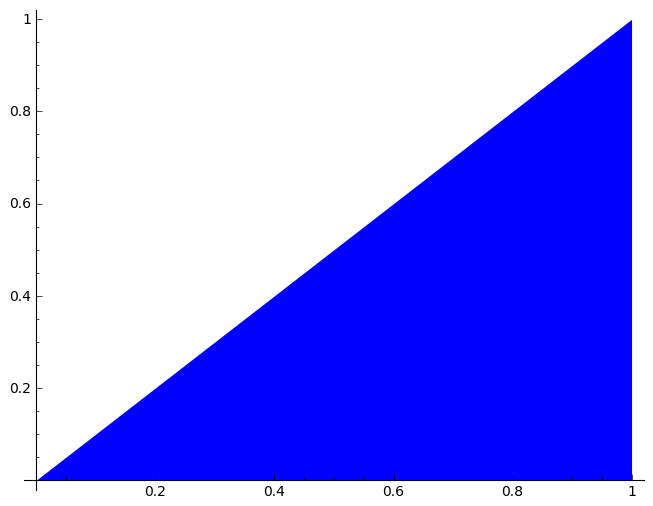} \\
$(x, x^2y)$ & $(x,xy^2)$ \\
\includegraphics[width=4.5cm]{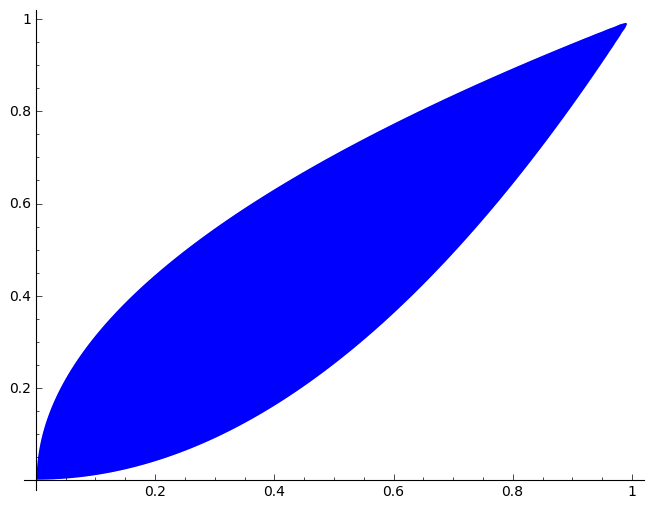} & \includegraphics[width=4.5cm]{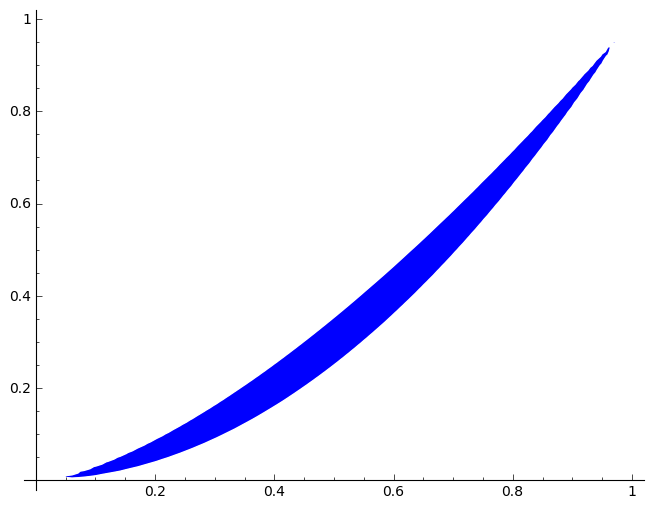} \\
$(x^2y,xy^2)$ & $(x^2y, x^3y^2)$ \\
\end{tabular}
\end{center}
\caption{Parameterizations of two-dimensional toric cubes inside the square $[0,1]^2$.} \label{fig:2d}
\end{figure}

We begin our discussion with an illustration of
toric cubes in dimension~$2$.

\begin{example}
Let $n=d= 2$ and consider monomial self-maps of the square:
$$ f \,:\, \,[0,1]^2 \,\rightarrow \,[0,1]^2 \,,\,\,\, (x,y) \,\mapsto \, (x^i y^j, x^k y^l ). $$
If $i,j,k,l > 0$ then the toric cube $\mathcal{C} = f\bigl([0,1]^2\bigr)$
is a region bounded by two monomial curves from $(0,0)$ to $ (1,1)$.
If precisely one of $i,j,k,l$ is zero then $\mathcal{C}$ will include two edges of $[0,1]^2$.
Some specific instances are shown in Figure~\ref{fig:2d}.
\end{example}

The image of the cube $[0,1]^d$ under an arbitrary polynomial map
is always a semialgebraic set, that is, it can be defined by a Boolean
combination of polynomial inequalities. This follows from
Tarski's Theorem on Quantifier Elimination \cite[\S 5.2]{BCR}.
However, such a semialgebraic set is usually not basic:
conjunctions do not suffice.
For a concrete example, consider the map
$$ [0,1]^2 \rightarrow \mathbb{R}^2, \,\,
(x,y) \mapsto  ( x+y,x+2y^2-2y ) . $$
Its image in $\mathbb{R}^2$ is not a basic semialgebraic set
since one edge of the square is mapped into the
interior of the image,
as seen in~Figure \ref{fig:eins}.

\begin{figure}
\begin{center}
\includegraphics[width=8cm]{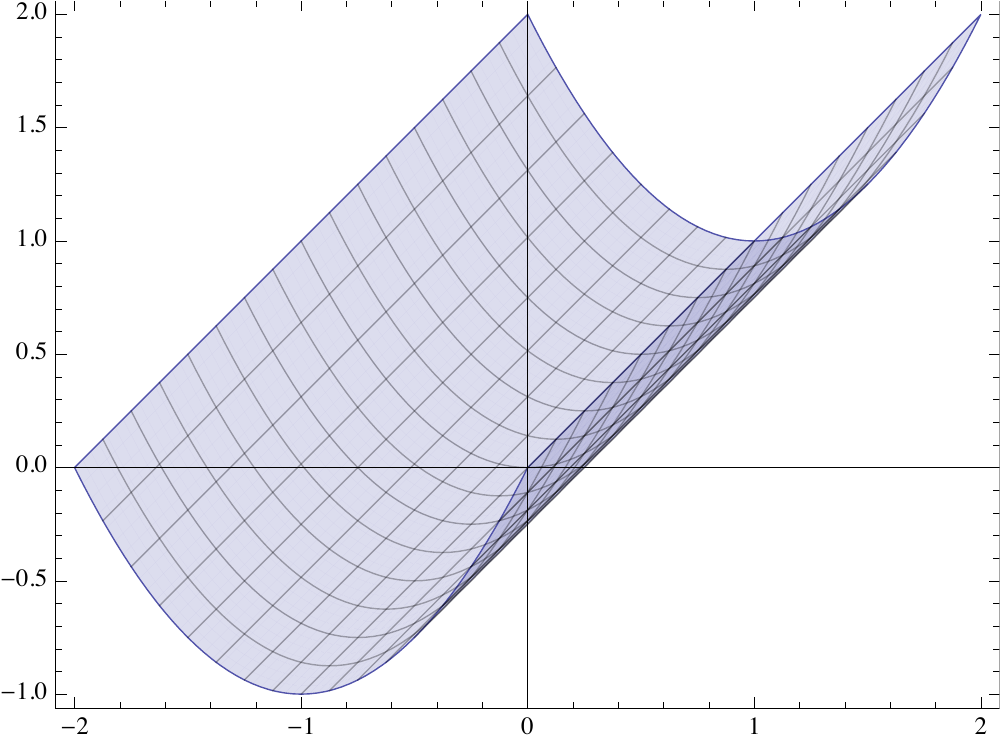}
\end{center}
\caption{This image of the square is not basic in $\mathbb{R}^2$.}
\label{fig:eins}
\end{figure}

Our first result ensures that such a folding never occurs for monomial maps.

\begin{proof}[Theorem \ref{thm:one}]
We first prove that monomial images of cubes are toric cubes.
Let $\mathcal{C}$ denote the image in $[0,1]^n$
of the map $f$ in (\ref{eq:alphamap}).
Write $\mathcal{C}^+ = \mathcal{C} \,\cap\,(0,1]^n$
for its subset of positive points.
Then $\mathcal{C}^+$ is non-empty,
and every point in $\mathcal{C}\backslash \mathcal{C}^+$
is the limit of a sequence of points in $\mathcal{C}^+$.
Let  ${\rm log}( \mathcal{C}^+)$ denote the image
of $\mathcal{C}^+$ in $ \mathbb{R}^n$ under the coordinatewise 
logarithm map, with reversed signs. The  cone ${\rm log}(\mathcal{C}^+)$ is 
 the image of the positive orthant
$\,{\rm log}((0,1]^d) =  \mathbb{R}^d_{\geq 0}\,$ under the linear map
$A : \mathbb{R}^d \rightarrow \mathbb{R}^n$, where
$A$ is the matrix whose rows are the vectors
${\bf a}_1,{\bf a}_2,\ldots,{\bf a}_n$.
In particular, ${\rm log}(\mathcal{C}^+)$ is a convex
polyhedral cone that is defined over $\mathbb{Q}$.
By the Weyl-Minkowski Theorem \cite{Zie}, we can write
the cone ${\rm log}(\mathcal{C}^+)$ as the solution set of
a finite system of linear inequalities of the form
\begin{equation}
\label{eq:linineq}
u_1 x_1 + u_2 x_2 + \cdots + u_n x_n \,\, \leq \,\, 
     v_1 x_1 + v_2 x_2 + \cdots + v_n x_n,
\end{equation}
where the $u_i$ and $v_j$ are non-negative integers.
By applying the exponential map, we conclude
that $\mathcal{C}^+$ is defined,
as a subset of $(0,1]^n$, by a finite set
of binomial inequalities.
Since $\mathcal{C} = \overline{ \mathcal{C}^+ } $,
the result follows from Lemma~\ref{lem:precube} below.

For the converse, we can reverse the reasoning
in the argument above. Suppose that $\mathcal{C} $ is a toric cube,
so it is the closure of  $\mathcal{C}^+ = \mathcal{C} \,\cap\,(0,1]^n$.
The cone ${\rm log}(\mathcal{C}^+)$ is defined by the linear inequalities
(\ref{eq:linineq}) corresponding to the binomial inequalities (\ref{eq:binoineq}) 
that define $\mathcal{C}$.
 We can write this cone as the image of 
some positive orthant $\mathbb{R}^d_{\geq 0}$ under some linear map.
That linear map is given by an integer matrix $A$ with
$d$ columns and $n$ rows ${\bf a}_1,{\bf a}_2,\ldots,{\bf a}_n$.
 The image of $(0,1]^d$ under the corresponding monomial
 map $f$ equals $\mathcal{C}^+$, and hence
 the image of the closed cube $[0,1]^d$ under $f$ is $\mathcal{C}$. \qed
 \end{proof}

 \begin{lemma} \label{lem:precube}
 Let $\mathcal{C}$ be a toric precube and 
 $\mathcal{C}^+ = \mathcal{C} \,\cap\,(0,1]^n$
 its subset of points with positive coordinates.
 Then the closure $\overline{\mathcal{C}^+}$ is a toric cube.
 \end{lemma}

\begin{proof}
Let $\mathcal{C}$ be a toric precube in $[0,1]^n$ that is defined by
a system of $N$ binomial inequalities (\ref{eq:binoineq}).
We present an algorithm that creates a finite list of
additional binomial inequalities such that the solution set
of the new enlarged system equals $\overline{\mathcal{C}^+}$.
Thus we give an algorithm for the {\em cubification of a precube}.

The procedure starts with the following step.
For each of the $N$ given inequalities (\ref{eq:binoineq}) we introduce a new
variable $\epsilon_i$ and we consider the binomial 
$$ x_1^{u_1} x_2^{u_2} \cdots x_n^{u_n} \,\, - \,\,
      x_1^{v_1} x_2^{v_2} \cdots x_n^{v_n} \cdot \epsilon_i. $$
Thus, we now have a collection of $N$ binomials in $N+n$ variables.
Let $I$ be the ideal in the polynomial ring generated by these binomials
and compute its saturation $I^{\rm sat}$ with respect to all 
unknowns.  We replace $I$ by the corresponding lattice
ideal $I^{\rm sat}$. Algorithmically, this corresponds to
computing a {\em Markov basis}, say, in the software {\tt 4ti2}.
For background on Markov bases see \cite[\S 1]{DSS}.

Since $I^{\rm sat}$ is a lattice ideal,
the complex variety $V(I^{\rm sat})$ is a finite union of toric varieties.
These components are
all orbit closures of the same torus action, and
only one of them intersects the positive orthant in $\mathbb{R}^{n+N}$.
The corresponding toric ideal $I^{\rm tor}$ is a
prime component of the radical ideal $I^{\rm sat}$.
Now, the toric variety $V(I^{\rm tor})$ is the closure of its
non-zero points. The same holds for the real points and
the non-negative points in the toric variety:
$$ V(I^{\rm tor}) \,\cap\, [0,1]^{N+n}  \,\, = \,\,
\overline{V(I^{\rm tor}) \,\cap \,(0,1]^{N+n}}.$$
The projection of this set onto the $n$ coordinates $x_1,\ldots,x_n$
is precisely the set  $\overline{\mathcal{C}^+}$. We obtain a
system of binomial inequalities that defines $\overline{\mathcal{C}^+}$
from the generators $\,{\bf x}^{\bf u} \epsilon^{\bf b} - {\bf x}^{\bf v} \epsilon^{\bf c}\,$
of the toric ideal $I^{\rm tor}$.
Namely, we take the inequality  ${\bf x}^{\bf u} \leq {\bf x}^{\bf v}$
if ${\bf b} = 0$ and ${\bf c} \not= 0$, we take the inequality
${\bf x}^{\bf u} \geq {\bf x}^{\bf v}$
if ${\bf b} \not= {\bf 0}$ and ${\bf c} = {\bf 0}$,
and we ignore the generator if both ${\bf b}$ and ${\bf c}$ are non-zero.
The resulting finite system of binomial inequalities shows that 
$\overline{\mathcal{C}^+}$ is a toric cube. \qed
\end{proof}

\begin{example}\label{pre-cube-example}
Let $\mathcal{C}$ denote the toric precube in $[0,1]^4$ defined by the inequalities
$$ a c \geq bd  \quad \hbox{and} \quad bc \geq a d . $$
This precube is not a toric cube because it contains the entire
face of points $(0,0,c,d)$ while every point in $\mathcal{C}^+$ satisfies $c \geq d$.
The toric cube $\overline{\mathcal{C}^+}$ is   cut out by the three inequalities
$\,a c \geq bd $,
$\, bc \geq a d $ and $ c \geq d$.
 To compute a parametric representation of $\overline{\mathcal{C}^+}$,
we take the negated logarithm and consider the cone
\begin{equation}
\label{eq:4cone}  \{ \,(A,B,C,D) \in \mathbb{R}_{\geq 0}^4 \,:\,  A+C \leq B+D, \,B+C \leq A+D\, \}. 
\end{equation}
This cone has the five extreme rays $ (1,1,0,0), (1,0,0,1),(0,1,0,1),(0,0,1,1)$ and $(0,0,0,1)$.
The $4 \times 5$-matrix with these columns specifies the linear map
$\,\mathbb{R}_{\geq 0}^5 \rightarrow \mathbb{R}_{\geq 0}^4$ whose
image is the cone (\ref{eq:4cone}). Writing the rows of that matrix as monomials,
we obtain the desired parametrization of the toric cube $\overline{\mathcal{C}^+}$:
$$ [0,1]^5 \rightarrow [0,1]^4 \,:\,(t_1,t_2,t_3,t_4,t_5) \mapsto (t_1 t_2 , t_1 t_3, t_4, t_2 t_3 t_4 t_5 ) = (a,b,c,d). $$
\end{example}

\section{Cell Decomposition}

In this section we study toric cubes through the lens of
topological combinatorics, and we prove  Theorem
\ref{thm:two}. Our task is to decompose a given toric cube as a CW-complex whose open cells are interiors 
of toric cubes with the further property that the boundaries of these cells are subcomplexes. 
We begin with the following basic observation concerning the topology of toric cubes.

\begin{remark}
Every toric cube is contractible. This is seen from the monomial parametrization $f$ as in (\ref{eq:alphamap}).
Namely, the map $g(t_1, \ldots, t_d,s)= f(s t_1, \ldots, s t_d)$ gives a
deformation retraction of the toric cube onto $\mathbf{0}$.
\end{remark}

To build a CW-complex from toric cubes
it is necessary to put toric cubes on the $0$-boundaries. To this end,
monomial maps like $x \mapsto (x,0,0)$ are allowed. This is consistent
with the previous definition after removing redundant zeros.
The singleton $\{\mathbf{0}\}$ is considered to be a toric
cube of dimension $0$. In this section toric cubes are mainly described 
by way of  their monomial parametrizations.

The CW-complexes treated in this text are
well-behaved, and we give a restricted definition that is
suitable for our purposes.
For more general versions see \cite{lundellWeingram}.
Let $D^m$ denote the closed $m$-dimensional disc, and let
$\partial D^m = S^{m-1}$ denote its boundary. Its interior, denoted
$\mathring{D}^m=D^m \backslash \partial D^m$, is an open $m$-cell.

\begin{definition}
An {\em $m$-dimensional  CW-complex} is a 
topological subspace $X^m$ of $\mathbb{R}^n$
that is constructed recursively in the following way:
\begin{itemize}
\item[(1)] If $m=0$ then $X^0$ 
is a discrete set of points.
\item[(2)] If $m>1$ then $X^m$ 
is given by the following data:
\begin{itemize}
\item[a.] an $(m-1)$-dimensional 
CW-complex  $X^{m-1}$ in $\mathbb{R}^n;$
\item[b.] a partition 
$\coprod_{\alpha \in I} \sigma_{\alpha}^m$ 
of $\,X^m \backslash X^{m-1}\,$
into open $m$--cells;
\item[c.] for every index $\alpha \in I,$ there is a
\emph{characteristic map} 
$\Phi_\alpha: D^m \rightarrow
 X^m$ such that 
 $\Phi_\alpha(\partial D^m) \subseteq X^{m-1}$ 
 and the restriction of $\Phi_\alpha$ to the open cell
 $\mathring{D}^m$ is a homeomorphism 
 with image $\sigma_\alpha^m.$
\end{itemize}
\end{itemize}
\end{definition}

One common way to identify a CW-complex 
for a space $X$ is to partition $X$
into open cells of different dimensions
and to give characteristic maps for each cell in that partition.
We demonstrate this for our running example.

\begin{example}\label{ex:toricCube}
Consider the toric cube $\mathcal{C}$ given by the 
monomial map $f(x,y,z) = (xy, yz, xz)$ in
Example \ref{ex:ex1}.
We define a CW-complex for $\mathcal{C}$ with $15$ cells by
\[
\begin{array}{lll}
f_{15} (x,y,z) =  (xy,yz,xz), & \mathcal{C}_{15}= f_{15}([0,1]^3), & \mathring{\mathcal{C}}_{15}= f_{15}((0,1)^3),  \\
f_{14} (x,y) =  (x,y,xy), & \mathcal{C}_{14}= f_{14}([0,1]^2), & \mathring{\mathcal{C}}_{14}= f_{14}((0,1)^2),  \\
f_{13} (x,y) =  (x, xy, y), & \mathcal{C}_{13}= f_{13}([0,1]^2), & \mathring{\mathcal{C}}_{13}= f_{13}((0,1)^2),  \\
f_{12} (x,y) =  (xy,x,y), & \mathcal{C}_{12}= f_{12}([0,1]^2), & \mathring{\mathcal{C}}_{12}= f_{12}((0,1)^2),  \\
f_{11} (x) =  (1,x,x), & \mathcal{C}_{11}= f_{11}([0,1]), & \mathring{\mathcal{C}}_{11}= f_{11}((0,1)),  \\
f_{10} (x) =  (x,1,x), & \mathcal{C}_{10}= f_{10}([0,1]), & \mathring{\mathcal{C}}_{10}= f_{10}((0,1)),  \\
f_{9} (x) =  (x,x,1), & \mathcal{C}_{9}= f_{9}([0,1]), & \mathring{\mathcal{C}}_{9}= f_{9}((0,1)),  \\
f_{8} (x) =  (0,0,x), & \mathcal{C}_{8}= f_{8}([0,1]), & \mathring{\mathcal{C}}_{8}= f_{8}((0,1)),  \\
f_{7} (x) =  (0,x,0), & \mathcal{C}_{7}= f_{7}([0,1]), & \mathring{\mathcal{C}}_{7}= f_{7}((0,1)),  \\
f_{6} (x) =  (x,0,0), & \mathcal{C}_{6}= f_{6}([0,1]), & \mathring{\mathcal{C}}_{6}= f_{6}((0,1)),  \\
& \mathcal{C}_{5}= \{  (1,1,1) \},  &  \mathring{ \mathcal{C}}_{5}= \{  (1,1,1) \},  \\
& \mathcal{C}_{4}= \{  (1,0,0) \},  &  \mathring{ \mathcal{C}}_{4}= \{  (1,0,0) \},  \\
& \mathcal{C}_{3}= \{  (0,1,0) \},  &  \mathring{ \mathcal{C}}_{3}= \{  (0,1,0) \},  \\
& \mathcal{C}_{2}= \{  (0,0,1) \},  &  \mathring{ \mathcal{C}}_{2}= \{  (0,0,1) \},  \\
& \mathcal{C}_{1}= \{  (0,0,0) \},  &  \mathring{ \mathcal{C}}_{1}= \{  (0,0,0) \}.  \\
\end{array}
\]
Here, the open cells of the CW-complex are $\mathring{ \mathcal{C}}_{1}, 
\ldots,  \mathring{\mathcal{C}}_{15}$ and the closed cells are 
$\mathcal{C}_{1}, \ldots,  \mathcal{C}_{15}.$ In Figure~\ref{fig:zwei}, 
the CW-complex is drawn with all open faces marked.
The closed cells, ordered by containment, form the face poset 
in Figure~\ref{fig:drei}. This poset is identical to the one seen in the
phylogenetic application \cite[Figure 1]{GLMS}.
\end{example}

The existence of a CW complex for toric cubes can be derived
from standard  theory \cite{lundellWeingram}. However,
being combinatorialists, we seek 
to find a small explicit one, ideally with the properties 
described in the following remark:

\begin{remark}\label{rmk:niceCW}
If the image of every $\partial D^n$ in the construction 
of a CW-complex is an $(n-1)$--dimensional topological manifold 
(that is, if each point has a neighborhood homeomorphic to the 
Euclidean $(n-1)$--dimensional space), then the CW-complex 
can be encoded combinatorially as the colimit of a diagram of 
spaces on a poset graded by dimension. This strategy was used by 
van Kampen in his thesis for combinatorial descriptions of
cell complexes, and is explained for CW-complexes in
\cite[Chapter 3]{lundellWeingram}. If the 
morphisms in this diagram are homotopic to constant maps, 
then its homotopy colimit is the nerve, which, if described 
as a simplicial complex, is the order complex of the 
aforementioned poset. The 
simplest case of this is when the image of each 
$\partial D^n$ 
is homeomorphic to a sphere. Such a CW-complex is
called \emph{regular}.
\end{remark}

In Example~\ref{ex:toricCube} the monomial map $f$
defining the toric cube and the characteristic map
were the same. However, in general this will not be the case,
e.g.~when the dimension of the image drops relative to the domain.
We typically need to subdivide.
This will be explained in  Example~\ref{ex:subdivide} and 
Proposition~\ref{prop:charMap}. 

Before proceeding,
we recall and introduce some notation.
The map $\log : (0,1]^d \rightarrow [0,\infty)^d$ is defined 
coordinate-wise by the negated logarithm.
This log map is a homeomorphism, and so is its 
inverse exp map. For a toric cube $\mathcal{C}$, its interior is 
best viewed in log-space. The closed polyhedral cone 
$\mathcal{D}=\log(\mathcal{C} \cap (0,1]^d )$ is 
full-dimensional in some lineality space and its interior 
is $\mathring{\mathcal{D}}.$ The interior of the toric cube 
$\mathcal{C}$ is $\mathring{\mathcal{C}}=\exp(\mathring{\mathcal{D}}).$ 
The dimensions of $\mathcal{C},\mathring{\mathcal{C}},\mathcal{D}$
 and $\mathring{\mathcal{D}}$ are all the same. For a 
 zero-dimensional toric cube, set $\mathring{\mathcal{C}}=\mathcal{C}.$

\begin{example}\label{ex:subdivide}
Let $d=4, n=3$ and
consider the toric cube given by the map
$$ f(t_1,t_2,t_3,t_4) \,\,\, = \,\,\, (t_1t_2t_4,t_2t_3,t_3t_4).$$
We want the interior of the image to be our only 
$3$-dimensional cell, but it is the image of an open $4$-cell under the monomial map $f$.
We cannot  use this map together with 
the $4$-dimensional cubical domain right off as a 
characteristic map since this map is not injective on the interior. The cone in log-space of this toric 
cube is spanned by the four rays
\[ r_1=(1,0,0), \,\, r_2=(1,1,0), \,\, r_3=(0,1,1), \,\, r_4=(1,0,1).
\]
A cross section of that cone is a quadrilateral with the vertices
corresponding to the rays in that order. 
The face poset of that quadrilateral has nine elements, labeled
$1,2,3,4,12,23,34,14$ and $1234$. To start constructing a 
characteristic map, we subdivide to get simplicial pieces. 
Our new rays will have the form $r_\sigma = \sum_{i\in \sigma} r_i.$ 
Considering the various $\sigma$ in the face poset $P$, we obtain:
\begin{equation}
\label{eq:ninerays}
\begin{array}{c}
r_{\{1,2,3,4\}}=(3,2,2), \\ 
r_{\{1,2\}}=(2,1,0), \, r_{\{2,3\}}=(1,2,1), \, r_{\{3,4\}}=(1,1,2),\, r_{\{1,4\}}=(2,0,1), \\ 
r_{\{1\}}=(1,0,0), \,\, r_{\{2\}}=(1,1,0), \,\, r_{\{3\}}=(0,1,1), \,\, r_{\{4\}}=(1,0,1). \\ 
\end{array}
\end{equation}
The rays corresponding to the maximal flags in $P$ span
simplicial cones that subdivide the cone spanned by the 
initial four rays. The $3 \times 9$-matrix with column vectors (\ref{eq:ninerays})
defines a new monomial map 
$\, \mathrm{sd}(f):[0,1]^P \rightarrow [0,1]^3\,$ by 
\[ (t_{\{1\}}, t_{\{2\}}, t_{\{3\}}, t_{\{4\}},
t_{\{1,2\}}, t_{\{2,3\}}, t_{\{3,4\}}, t_{\{1,4\}},
t_{\{1,2,3,4\}}) \quad \mapsto
\]
\[\begin{array}{l}
( \, t_{\{1\}} t_{\{2\}} t_{\{4\}} t_{\{1,2\}}^2 t_{\{2,3\}} t_{\{3,4\}} t_{\{1,4\}}^2 t_{\{1,2,3,4\}}^3, \\
\,\,\, t_{\{2\}} t_{\{3\}}t_{\{1,2\}} t_{\{2,3\}}^2 t_{\{3,4\}} t_{\{1,2,3,4\}}^2, \,\,\,\,
\,\, t_{\{3\}} t_{\{4\}} t_{\{2,3\}} t_{\{3,4\}}^2 t_{\{1,4\}}t_{\{1,2,3,4\}}^2).
\end{array}\]
By setting $\, t_1 = t_{\{1\}} t_{\{1,2\}} t_{\{1,4\}} t_{\{1,2,3,4\}} $
and similarly for $t_2,t_3,t_4$, we see that $f$ and $\mathrm{sd}(f)$
have exactly the same image. Thus, they define the same toric cubic.
Guided by the simplicial subdivision above, we define 
$D \subset [0,1]^P$ to be
\begin{center}
\includegraphics{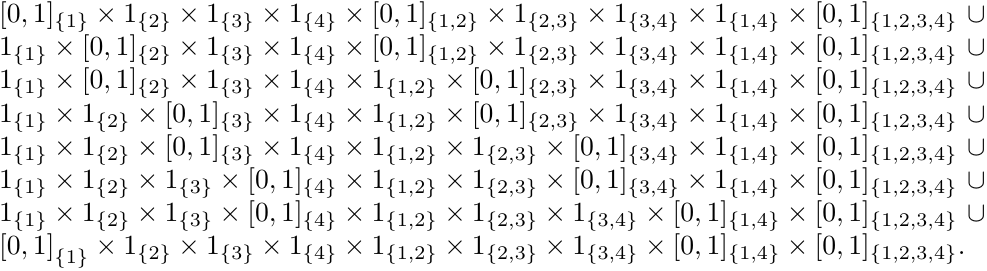}
\end{center}
From the subdivision we can derive that
$\,\mathrm{sd}(f)(D)=\mathrm{sd}(f)([0,1]^P)$ $=f([0,1]^4).$ 
The domain $D$ is $3$-dimensional, as is the toric cube, and 
one can see that the restriction of $\mathrm{sd}(f)$ to the relative interior of $D$
is a homeomorphism onto the interior of 
the toric cube, as required for the characteristic maps. 
What remains to be shown at this point is that $D$ is in fact
a $3$-dimensional ball. This is true, and we present
a general argument in the proof of the next proposition.
\end{example}

\begin{proposition}\label{prop:charMap}
Let $\mathcal{C}\subset [0,1]^n $ be a toric cube and consider the convex polytope
\[ \mathcal{P} \,\,= \,\,\bigl\{ \,\mathbf{y} \in \log( \mathcal{C} \cap (0,1]^n ) \mid \mathbf{y} \cdot \mathbf{1} \leq 1 
\bigr\} .\]
There exists a continuous map $\Phi : \mathcal{P} \rightarrow \mathcal{C}$ 
whose restriction to the interior of $\mathcal{P}$ is a 
homeomorphism onto the interior of $\mathcal{C}$, with the property that the restriction
of $\Phi$  to the boundary of $\mathcal{P}$ maps onto the
 boundary of  $\mathcal{C}.$
\end{proposition}

\begin{proof}
The cone $\mathcal{D}=\log( \mathcal{C} \cap (0,1]^n)$ 
is spanned by some non-negative integer
rays $r_1,r_2,\ldots, r_d$.
The toric cube is the image of  the monomial map
$f(t_1,t_2,\ldots,t_d)=(m_1,m_2,\ldots,m_n)$ where
$m_j=\prod_{i=1}^d t_i^{r_i \cdot e_j}$ and $e_j$ is the
$j$th unit vector. Without loss of generality we may 
assume that $\,r_1\cdot \mathbf{1} = \cdots = r_d \cdot \mathbf{1}$.

The non-empty subsets $S$ of $\{1,2,\ldots, d\}$ such that
$\{ r_i \mid i\in S \}$ is a minimal set of spanning rays of
a face of $\mathcal{D}$, ordered by inclusion, is a poset $P.$
This poset $P$ is isomorphic to the face poset of $\mathcal{D}$
minus the minimal element.

We fix rays $\,r_\sigma=\sum_{i \in \sigma} r_i\,$ for each $\sigma \in P$,
and we define a monomial map $\,\mathrm{sd}(f):[0,1]^P\rightarrow [0,1]^n\,$
by sending $\,(t_\alpha)_{\alpha \in P}\,$ to 
$\,(m_1^{\mathrm{sd}},m_2^{\mathrm{sd}},\ldots,m_n^{\mathrm{sd}})\,$
where
\[m_j^{\mathrm{sd}} \,\, = \,\,
\prod_{\sigma\in P} t_\sigma^{  \sum_{i\in \sigma} r_i   \cdot e_j}.
\]
The cone in log-space  defined by
$\mathrm{sd}(f)$ is the same as the one for $f,$ but we have
introduced rays that barycentrically subdivide it.
The simplicial cones in that subdivision are indexed by the set
of maximal chains in $P$. We define
\[ D \quad = \bigcup_{\hbox{$C$ maximal chain of $P$}} 
\quad\prod_{\sigma \in P} 
\quad
\left\{ 
\begin{array}{cl}
\, [0,1] & \quad {\rm if} \,\, \sigma \in C, \\
\, \{1\} & \quad {\rm if} \,\, \sigma \not\in C.
\end{array}
\right.  \]
The barycentric subdivision ensures that
$\mathcal{C}=f([0,1]^d)=\mathrm{sd}(f)(D)$. Moreover,
the restriction of $\mathrm{sd}(f)$ to the interior of
$D$ is a homeomorphism onto $\mathring{\mathcal{C}}$.

We now construct a homeomorphism between $D$ and the polytope. We first fix
the antipodal map, componentwise defined by 
$t \rightarrow 1-t,$ to get a 
homeomorphism $a:D'\rightarrow D$ for 
\[ D' \quad =\bigcup_{\hbox{$C$ maximal chain of $P$}}
\quad\prod_{\sigma\in P}
\quad
\left\{
\begin{array}{cl}
\, [0,1] & \quad {\rm if} \,\, \sigma \in C, \\
\, \{ 0\} & \quad {\rm if} \,\, \sigma \not\in C.
\end{array}
\right. \]
Let $\Delta$ be the simplex in $\mathbb{R}^P$ spanned
by the unit vectors and the origin. There is a standard
homeomorphism from $\Delta$ to $[0,1]^P$ which maps
$\,\mathbf{t}=(t_\sigma)_{\sigma \in P} \neq \mathbf{0}\,$ to 
$\,\frac{\sum_{\sigma \in P}t_\sigma}{\max_{\sigma \in P}t_\sigma}\mathbf{t}\,$
and $ \mathbf{0}$ to $\mathbf{0}$. Restricting this
map to get the image $D'$, we obtain a homeomorphism
$b:D'' \rightarrow D'$ for $D''$ as follows.

By construction, $D''$ is the cone with apex
$\mathbf{0}$ over the standard realization of the
order complex of the poset $P$. Since $P$ is the face poset (minus the minimal element)
of the polyhedral cone $\mathcal{D}$, we have a homeomorphism
$c$ from the polytope $\{ \mathbf{y} \in \mathcal{D} \mid \mathbf{y} \cdot \mathbf{1} \leq 1 \}=
\{ \mathbf{y} \in \log( \mathcal{C} \cap (0,1]^n ) \mid \mathbf{y} \cdot \mathbf{1} \leq 1 \}$
to $D''$. The composition of these maps gives 
a map $\Phi$ satisfying the requirements.  \qed
\end{proof}

Consider a toric cube $\mathcal{C}\subseteq [0,1]^n$
and set, for each $I\subseteq \{1,2, \ldots, d \}$,
\[ D_I \,\,= \,\,\prod_{i=1}^d 
\left\{
\begin{array}{cl}
(0,1] & \,\,\textrm{if }i\in I,\\
\{0\} & \,\, \textrm{if }i\not\in I.\\
\end{array}
\right.
\]
The various choices for $I$ allow for the various possibilities for which unknowns among
$t_1,\dots ,t_d$ are strictly positive, with all those not in $I$ being 0.  This choice in turn
dictates which of the parameters derived under barycentric subdivision (as in Example~\ref{ex:subdivide}) are strictly positive and which are $0$. This distinction enables us to embed  $D_I$ into $[0,1]^n$.

For each $I,$ the set $\mathcal{C}\cap D_I$ can be
partitioned into open cells according to the open
cells of its polyhedral cone in log-space, namely the space in which we take lots of the 
nonzero unknowns. The collection
of these open cells is the \emph{Tuffley partition} of
$\mathcal{C}$. The name refers to Christopher Tuffley, whose Masters Thesis, written under the
supervision of Mike Steel, predated \cite{GLMS,MS}.

Unlike in Example~\ref{ex:toricCube},
the Tuffley partition does not always give the
open cells of a CW-complex. 
The main problem is that the boundary of a $d$--dimensional
open cell might intersect a $d'$-dimensional open cell
with $d' \geq d.$ To solve this problem, subdivisions are required.
Moving into log-space, one sees that all  peculiarities
of polyhedral subdivisions are present, but also that
the algorithmic tools from that area are readily accessible
to address them.

\begin{proposition}\label{prop:capTwoIntToricCubes}
If  $\mathcal{C}_1$ and $\mathcal{C}_2$ are toric cubes in $[0,1]^n$,
then there exists a third toric cube $\mathcal{C}_3$ in $[0,1]^n$ such that 
 $\,\mathring{\mathcal{C}}_1 \cap \mathring{\mathcal{C}}_2 =  \mathring{\mathcal{C}}_3.$
\end{proposition}

\begin{proof}
This is immediate from Theorem~\ref{thm:one}.
The union of the binomial inequalities defining $\mathcal{C}_1$
and those defining $\mathcal{C}_2$ specifies a toric precube. If we take
$\mathcal{C}_3$ to be the cubification of that precube, then
$\mathcal{C}_3$ has the desired properties. \qed
\end{proof}

\begin{lemma}\label{prop:intToricCubes}
If $\mathcal{C}, \mathcal{C}' \subseteq [0,1]^n $ are 
toric cubes with $\mathring{\mathcal{C}}\subseteq \mathring{\mathcal{C}}',$  
then there exist toric cubes 
$\mathcal{C}_1,  \mathcal{C}_2, \ldots,  \mathcal{C}_k$ 
such that $\mathcal{C}=\mathcal{C}_1$ and $  \mathring{\mathcal{C}}'= \cup_{i=1}^k \mathring{\mathcal{C}}_i$
 and $\mathring{\mathcal{C}}_i \cap \mathring{\mathcal{C}}_j = \emptyset$ for
$i \neq j $.
\end{lemma}

\begin{proof}
Let $\mathcal{D}$ and $\mathcal{D}'$ 
be the convex polyhedral cones in log-space that correspond to the toric cubes $\mathcal{C}$ 
and $\mathcal{C'}$. The cone $\mathcal{D}$ is contained 
in $\mathcal{D}'$ and there is a subdivision of the cone 
$\mathcal{D}'$ into cones $\,\mathcal{D}_1, \mathcal{D}_2,
 \ldots, \mathcal{D}_k$ such that $\mathcal{D} = \mathcal{D}_1$.
 Let $\mathcal{C}_1, \mathcal{C}_2, \ldots, \mathcal{C}_k$ denote the
 corresponding toric cubes. The required properties 
 follow directly from the fact that  log and exp
 are homeomorphisms. \qed
\end{proof}

Before proving Theorem \ref{thm:two}, we need one lemma
that takes into account the subdivisions of cells added in the process of building
the CW-complex.

\begin{lemma}\label{prop:buildCW}
Let $X$ be a CW-complex 
whose open cells are interiors of toric cubes,
and $\mathcal{C}_1, \mathcal{C}_2, \ldots,  \mathcal{C}_r$ 
further toric cubes, all embedded in a common unit cube. 
There is a CW-complex $\tilde{X}$ whose open 
cells are interiors of toric cubes, such that
 each open cell of $X$
 is a disjoint union of open cells in  $\tilde{X}$, 
 and $\sigma \cap \mathring{\mathcal{C}}_i 
 \in \{ \emptyset, \sigma\}$ for each open cell $\sigma$ of $\tilde{X}$.
\end{lemma}

\begin{proof}
It suffices to show this for $r=1$ and 
 repeat the argument. Set
$\mathcal{C}=\mathcal{C}_1$.
The proof is by induction on the dimension $m$ of $X$.
If $m=0$ then we are done:
the intersection of a point and an 
 open cell is either empty or that point.

If $m>0$, then we use Lemma~\ref{prop:intToricCubes}
 to subdivide all $m$--dimensional open cells of $X$ 
 such that their intersection with $\mathring{\mathcal{C}}$ 
 is either empty or the open cell itself. Let
  $\mathring{\mathcal{C}}'_1, \ldots, \mathring{\mathcal{C}}'_t$ 
  be the new open cells and all open cells 
  on their boundaries given by the Tuffley partition. 
  Now, by induction, apply Lemma~\ref{prop:buildCW}
  to the $(m-1)$--skeleton of $X,$ with 
  the collection $\mathcal{C}$ and 
  $\mathcal{C}'_1, \ldots, \mathcal{C}'_t$ to 
  refine, to get an $(m-1)$--dimensional 
  CW-complex $X'$. We extend 
  $X'$ to $\tilde{X}$ by adding on the new
   open cells $\mathring{\mathcal{C}}'_1, \ldots, \mathring{\mathcal{C}}'_t$ we just constructed. Note that the 
   open cells added  from $X'$
    to $\tilde{X}$ need not  be 
    $m$--dimensional, but they cannot be
     on the boundary of anything in $X'$
      since they are in open $m$-cells of $X$. \qed
\end{proof}

\begin{proof}[of Theorem \ref{thm:two}]
Let $\mathcal{C}$ be a toric cube in $[0,1]^n$. The 
\emph{support} of a point in $[0,1]^n$ is the set of its strictly positive 
coordinates. Let $S_1, S_2, \ldots, S_t$ be a linear ordering 
of the subsets of $\{1,2, \ldots, n\}$ that each support a point in 
$\mathcal{C}$, where $i < j$ whenever $S_i \subset S_j$.
 Let $\mathcal{C}_k$ be the points 
in $\mathcal{C}$ with support $S_k$, and
$\mathring{\mathcal{C}}_k^1, \mathring{\mathcal{C}}_k^2, \ldots, \mathring{\mathcal{C}}_k^{s_k}$
 the open sets in the Tuffley partition of $\mathcal{C}$ 
whose union is $\mathcal{C}_k$.

We start building from the point $\mathcal{C}_1=\mathbf{0}$ to 
get the CW-complex $X_1.$
Next we will build a CW-complex $X_k$ on $\cup_{i=1}^k \mathcal{C}_i$
for every $k=2,3,\ldots,t.$ Note that this filtration is \emph{not}
by dimension, but rather by a linear extension of the set inclusion order
on the different supports.

 For $k=2,3, \ldots, t$, we proceed as follows:

\begin{itemize}
\item[(1)] 
A point on the boundary of a cell $\mathring{\mathcal{C}}_k^i$ 
is in $\mathcal{C}_k$ if the point and $\mathcal{C}_k$ have the 
same support. Otherwise the support of that point is smaller 
than that of $\mathcal{C}_k,$ and the point is in the CW-complex $X_{k-1}.$
\item[(2)] 
Let $\mathring{\mathcal{C}}_1,  \ldots, \mathring{\mathcal{C}}_{t_k}$ 
be the boundary cells of 
$\mathring{\mathcal{C}}_k^1, \mathring{\mathcal{C}}_k^2, \ldots, \mathring{\mathcal{C}}_k^{s_k}$ 
that have smaller support than $\mathcal{C}_k$. 
Now use Lemma~\ref{prop:buildCW} to subdivide the open 
cells of the CW-complex $X_{k-1}$ with respect to  
$\mathring{\mathcal{C}}_1, \ldots, \mathring{\mathcal{C}}_{t_k}$ 
to get the CW-complex $\tilde{X}_{k-1}.$ 
Any open cell on the boundary of a $\mathring{\mathcal{C}}_k^i$
whose support drops is a union of open cells in
$\tilde{X}_{k-1}.$ If the support doesn't drop, coherent
boundary maps are inherited from the log-cone of $\mathcal{C}_k.$
\item[(3)] We extending the CW-complex $\tilde{X}_{k-1}$
by the open cells
$\mathring{\mathcal{C}}_k^1, \mathring{\mathcal{C}}_k^2, \ldots, \mathring{\mathcal{C}}_k^{s_k}$.
By construction, their boundaries are subcomplexes. This
defines a CW-complex $X_k$ whose open cells 
are $\cup_{i=1}^k \mathcal{C}_i.$
\end{itemize}
The desired CW-complex $X_t$ for $\mathcal{C}=\cup_{i=1}^t \mathcal{C}_i$ has now been constructed. \qed
\end{proof}

\end{document}